\documentclass[11pt]{amsart}
\usepackage{lineno}
\usepackage{amsmath}
\usepackage{amssymb}
\usepackage{url}
\usepackage{amscd}
\usepackage{texdraw}
\usepackage{color}
\usepackage{graphicx}
\usepackage{tikz}
\usepackage{tikz-3dplot}
\usepackage{mathtools}

\numberwithin{equation}{section}

\theoremstyle{plain}
\newtheorem{definition}{Definition}
\newtheorem{thm}{Theorem}

\newtheorem{lem}{Lemma}

\newtheorem{pro}{Proposition}

\newtheorem{cor}{Corollary}
\newtheoremstyle{step}{}{}{}{}{}{:}{ }{}
\theoremstyle{step}

\newenvironment{case}[1]{\par\noindent\textit{Case #1:}}{\par}

\def\leq{\leqslant}
\def\geq{\geqslant}

\def\Ra{\Rightarrow}

\newcommand{\abs}[1]{\lvert#1\rvert}
\newcommand{\Ss}{\Bbb{S}^1}

\newcommand{\R}{\Bbb{R}}
\newcommand{\D}{\mathbb{D}}
\newcommand{\Dc}{\Bar{\Bbb{D}}}

\newcommand{\C}{\Bbb{C}}

\newcommand{\DeltaEFn}{\Delta(E_n,F_n)}
\newcommand{\DeltaEFnprime}{\Delta(E_n^\prime,F_n^\prime)}
\renewcommand{\bar}[1]{\overline{#1}}
\newcommand{\xnp}{x_n'}
\newcommand{\anp}{a_n'}
\newcommand{\bnp}{b_n'}
\newcommand{\dnp}{d_n'}
\newcommand{\anr}{\hat{a}_n}
\newcommand{\bnr}{\hat{b}_n}
\newcommand{\arc}[1]{\overset{\frown}{#1}}

\begin{document}

\title
{On Asymptotically symmetric embeddings and conformal maps}
\author{Ylli Andoni}
\address{ Department of Mathematics, Emory University, Atlanta, GA, 30322 USA}
\email{ylli.andoni@emory.edu}

\author{Shanshuang Yang}
\address{ Department of Mathematics, Emory University, Atlanta, GA, 30322 USA}
\email{syang05@emory.edu}

\begin{abstract}
This paper is devoted to the study of conformal maps of the unit disk $\D$ in the plane onto a bounded Jordan domain $G$. The main aim is to show that such a map is asymptotically symmetric if and only if $G$ is bounded by a symmetric quasicircle.
\end{abstract}
\maketitle

\section{Introduction}
Let $f$ be a conformal map from the unit disk $\D$ in the complex plane $\C$ onto a bounded Jordan domain $G$. There is a rich theory about the interplay between the analytic properties of the map $f$ and the geometric properties of the boundary Jordan curve $\partial G$ (see \cite{pommerenke2013boundary} for many classical results). For example, one can use quasisymmetry of $f$ to characterize a quasicircle as follows.

\vspace{0.5cm}
{\bf Theorem A.} {\it Let $f: \D\rightarrow G$ be a conformal map. Then the following conditions are equivalent.\\
(a) $J=\partial G$ is a quasicircle;\\
(b) $f$ is quasisymmetric on $\D$;\\
(c) The homeomorphic extension of $f$ to the boundary is quasisymmetric.}

\vspace{0.5cm}
Recall that a Jordan curve $J$ is called a {\it quasicircle} if it is the image  of the unit circle $\Ss$ under a quasiconformal map of $\C$ onto itself. This class of Jordan curves has been extensively studied and dozens of characterizations have been found across different areas of mathematics (see \cite{gehring2012ubiquitous} and the references therein). One characterization is the following simple geometric property, often adopted as the geometric definition (see \cite[Definition~5.4]{pommerenke2013boundary}):
A Jordan curve $J$ is a quasicircle if and only if there exists a constant $M$ such that 
\begin{equation}
    {\textbf{diam}}J(a,b)\leq M|a-b|
\end{equation}
for all $a,b\in J$, where (and in what follows) $J(a,b)$ denotes the arc of smaller diameter of $J$ between $a$ and $b$. \par
Moreover, a {\it quasisymmetric embedding} abbreviated {\it QS} is an embedding $f\colon X\to Y$ between metric spaces such that there exists an increasing homeomorphism $\eta\colon [0,\infty)\to [0,\infty)$ so that, 
\begin{equation}
    \frac{\abs{x-a}}{\abs{x-b}}\leq t\implies \frac{\abs{f(x)-f(a)}}{\abs{f(x)-f(b)}}\leq \eta(t)
\end{equation}
for all distinct points $x, a, b\in X$. 
The concept of quasisymmetry was introduced by Ahlfors and Beurling in their study of boundary behavior of qusiconformal maps of the upper half plane onto itself \cite{boundary_correspondence}. The above general definition of quasisymmetry in a metric space setting is due to Tukia and V\"ais\"al\"a \cite{tukia1980quasisymmetric}.
For more on the theory of quasisymmetric maps, from the most general point of view offered by metric spaces, we refer the reader to \cite{heinonen2001lectures}.

For a brief moment we return our attention to Theorem A once again. We note that Theorem A is a combination of several well-known results. The equivalence of (a) and (c) follows from \cite[Proposition~5.10]{pommerenke2013boundary} and \cite[Theorem~5.11]{pommerenke2013boundary}. The equivalence of (a) and (b) follows from \cite{tukia1980quasisymmetric}. See \cite[Theorem~3.12]{MR1057080} for a higher dimensional version of (c)$\Rightarrow$(b).

Motivated by the above characterization of quasicircles, the main purpose of this paper is to use the {\it asymptotic symmetry} property of conformal maps to characterize a special subclass of quasicircles, namely, {\it symmetric quasicircles}. 
Following \cite[Section~11.2]{pommerenke2013boundary}, a Jordan curve $J$ is called a {\it symmetric quasicircle} (or an {\it asymptotically conformal curve}) if 
\begin{equation}
    \max_{w\in J(a,b)}\frac{\abs{a-w}+\abs{w-b}}{\abs{a-b}}\to 1
\end{equation}
as $\abs{a-b}\to 0$. Further, we call a Jordan domain $G$ that is bounded by a symmetric quasicircle, a {\it symmetric quasidisk}. 

The notion of \textit{AS} embeddings, which we will use to characterize symmetric quasicircles, was first introduced in \cite{brania2004asymptotically} in its full generality, although weaker notions appear as early as in \cite{MR215986}. This class of mappings was later studied by Gardiner-Sullivan, see \cite{gardiner1992symmetric} for more applications of this. Therein the authors define a \textit{symmetric} map of the real line $\R$, as a quasisymmetric homeomorphism $f$, for which 
\[
\lim_{t\to0^+}\frac{f(x+t)-f(x)}{f(x)-f(x-t)}=1
\]
uniformly for all $x\in \R.$ Since then, this subclass of quasisymmetric homeomorphisms of the real line $\R$ has also played a significant role in the study of Teichm\"uller spaces (see for example \cite{MR3834655}). One can show that, similar to the Beurling-Ahlfors $M$-condition, this subclass of quasisymmetric homeomorphisms appears naturally as the boundary correspondence of quasiconformal maps of the upper half-plane onto itself that are {\it asymptotically conformal} on the real line (see \cite{MR2560683} for an argument). 

Recall that a homeomorphism $f: D\rightarrow G$ between domains in $\C$ is called {\it $K$-quasiconformal} (or $K$-QC) if it has locally square integrable partial derivatives with bounded local dilatation:
\[
K(f,z)=\frac{|f_z|+|f_{\bar{z}}|}{|f_z|-|f_{\bar{z}}|}\leq K
\]
for almost all $z\in D$. Furthermore, a quasiconformal map $f:\C\rightarrow\C$ is called {\it asymptotically conformal} on the unit circle if its local dilatation approaches $1$ as $|z|\rightarrow 1$, more precisely, if for any 
$\epsilon>0$ there exists $\delta>0$ such that $f$ is $(1+\epsilon)$-QC in the $\delta$-neighborhood of $\Ss$: $\{z: 1-\delta<|z|<1+\delta\}$. 
And asymptotical conformality on the real line can be defined similarly via a M\"obius transformation. 

As we stated above, the main purpose of this article is to characterize symmetric quasicircles in terms of the \textit{asymptotic symmetry} property of conformal maps, which is defined below. 
\begin{definition}[$AS$]
    An embedding $f\colon X\to Y$ between metric spaces is called asymptotically symmetric, or abbreviated $AS$ if for all $\epsilon>0$ and $t>0$ there exists a $\delta>0$ such that for all distinct points $x,a,b\in X$ contained in a ball of radius $\delta$
    \begin{align}\label{AS:property}
        \frac{\abs{x-a}}{\abs{x-b}}\leq t\implies \frac{\abs{f(x)-f(a)}}{\abs{f(x)-f(b)}}\leq (1+\epsilon)t.
    \end{align}
\end{definition}
Comparing the AS condition (1.4) with the QS condition (1.2), one notes that (1.4) is a localized but strengthened version of (1.2) by replacing $\eta(t)$ with $(1+\epsilon)t$. 
In the spirit of Theorem A, our main goal here is to establish the following characterizations for symmetric quasicircles. 
\begin{thm}\label{thm:main_theorem}
Let $f: \D\rightarrow G$ be a conformal map. Then the following conditions are equivalent.\\
(a) $J=\partial G$ is a symmetric quasicircle;\\
(b) $f$ is asymptotically symmetric on $\D$;\\
(c) The homeomorphic extension of $f$ to the boundary is asymptotically symmetric.
\end{thm}
Note that the equivalence of (a) and (c) was proved in \cite{brania2004asymptotically}, using the modulus of the Teichm\"uller ring domain and properties associated with it. The main focus of this article will be to show that (a) implies (b) and (b) implies (c). More characterizations of symmetric quasicircles can be found in \cite{MR3771265}, \cite{MR1314950},   \cite{wu2000symmetric}.\par
The first implication (a) $\Rightarrow$ (b) will be one of the central results of this paper. It involves the use of analytic properties of conformal mappings from the unit disk onto a symmetric quasidisk, and the use of the modulus of the Teichm\"uller ring domain. Therefore we devote Section $2$ to introducing these tools and discussing how these tools are to be used. In Section $3$, we will jump right into working out the details for the first implication. This can be found under Theorem 2. \par
It would then still remain to show that (b) implies (c). This will be a simple consequence of a limiting argument on the boundary of the unit disk. For completeness, this will be formulated in Theorem 3.\par
In Section $4$, we conclude this paper with some further comments and open questions on asymptotically symmetric embeddings of the unit disk. 
\section{Preliminaries}
As mentioned above, our main results use analytic techniques, modulus estimates from the theory of quasiconformal mappings, and some classical Euclidean geometry. The purpose of this section is therefore to provide a detailed overview of the analytical and modulus estimates tools needed in Section~\ref{Section:3}.

\subsection{A theorem of Pommerenke}
We start the discussion by citing a collection of results, which can be found in \cite{pommerenke2013boundary}.
\begin{lem}\cite[Theorem 11.1]{pommerenke2013boundary}\label{lemma:pommerenke}
    Let $f$ map $\D$ conformally onto the inner domain of the Jordan curve $J$. Then the following conditions are equivalent
    \begin{enumerate}
        \item $J$ is a symmetric quasicircle.
        \item $\frac{f(z)-f(\zeta)}{(z-\zeta)f^{\prime}(z)}\to 1$ as $|z|\to1^-$ uniformly for $\zeta\in\Dc$ with $\frac{|z-\zeta|}{1-|z|}\leq a$ for some $a>0.$
        \item $f$ has a quasiconformal extension to $\C$ that is asymptotically conformal on the unit circle. 
    \end{enumerate}
\end{lem}
\subsection{The Teichm\"uller function $\Psi(t)$} 
The second main ingredient used in the proof of our main result is the Teichm\"uller function and the Teichm\"uller ring domain. The Teichm\"uller ring domain $R_{T}(t)$ for $t>0$ is the domain whose complement consists of the two disjoint continua $E=[-1,0]$ and $F=[t,\infty]$, lying on the extended real axis. The curve family that connects the two disjoint continua $E,F$ in $\C$ will be denoted by $\Delta(E,F)$, and the modulus of the curve family can be expressed as 
\[
M(\Delta(E,F))=\Psi(t)
\]
where $\Psi(t)$ is the Teichm\"uller function, which is continuous and strictly decreasing with 
\[
\Psi(0)=\lim_{t\to 0}\Psi(t)=\infty \quad {\text {and}} \quad \Psi(\infty)=\lim_{t\to \infty}\Psi(t)=0.
\]
Another fact about the Teichm\"uller function that will be of use to us is the following comparison principle. If $E,F$ are two disjoint continua in $\C$ with $a,b\in E$ and $c,d\in F$ then 
\[
M(\Delta(E,F))\geq \Psi(t)
\]
with $t$ the cross ratio of $a,b,c,d$ being given as $$t=[a,b,c,d]=\frac{\abs{b-c}}{\abs{a-b}}\frac{\abs{a-d}}{\abs{c-d}}.$$
\par
For an introduction to the theory of quasiconformal mappings, the modulus of curve families and their applications, see the classic reference by Vaisala \cite{vaisala2006lectures} which has stood the test of time or the more modern one by Gehring, Martin, Palka \cite{gehring2017introduction}.

\subsection{A cross ratio estimate using modulus}
Since the Teichm\"uller ring domain and Teichm\"uller function, in conjunction with the quasiconformal extension, will be used repeatedly to estimate some cross-ratios in our proof, we explicitly formulate the following lemma, which may be of independent interest and is crucial in the proof of (a) $\Rightarrow$ (b) in Theorem 1. It essentially states that the cross ratio of concyclic points is preserved in an infinitesimal sense under an asymptotically conformal embedding. Recall that a homeomorphism $f$ of $\C$ is called {\it asymptotically conformal} on the unit disk if 
\[
K(f,z)=\frac{|f_z|+|f_{\bar{z}}|}{|f_z|-|f_{\bar{z}}|}\rightarrow 1
\]
as $|z|\rightarrow 1$. See \cite{gutlyanskii1995local}, \cite{MR1608159}, \cite[Section~11]{pommerenke2013boundary} and the references therein for more on asymptotically conformal mappings and their properties.

\begin{lem}[Cross-ratio estimate]\label{lem:Modulus_estimates}
Let $f$ be a homeomorphism of $\C$ that is asymptotically conformal on the unit circle. For each $n$, let $x_n,a_n,b_n$ be distinct points in $\C$ and $C_n$ denote the unique circle (or straight line) determined by these points. 
Assume that the sequences $x_n,a_n,b_n$ converge to a common limit point $x\in\Ss$ and that the diameter of the component of $C_n-\{a_n,b_n\}$ containing $x_n$, denoted by $C_n(x_n)$, converges to zero. Furthermore, let $d_n$ be a fourth point on $C_n\backslash C_n(x_n)$, and if $C_n$ is a straight line we let $d_n=\infty$.
If the limit 
\[
\lim_{n\to\infty}\frac{\abs{x_n-a_n}}{\abs{x_n-b_n}}\frac{\abs{d_n-b_n}}{\abs{d_n-a_n}}= \tau
\]
exists with $0<\tau<\infty$, then 
\[
\lim_{n\to\infty}\frac{\abs{f(x_n)-f(a_n)}}{\abs{f(x_n)-f(b_n)}}\frac{\abs{f(d_n)-f(b_n)}}{\abs{f(d_n)-f(a_n)}}= \tau.
\] 
\end{lem}

\begin{proof}
    Let $x_n,a_n,b_n, d_n$ be sequences of points as described in the Lemma with $x_n,a_n,b_n$ converging to the same limit point $x\in \Ss$. 
    Denote the cross-ratio of $\{x_n,a_n,b_n,d_n\}$ and the cross-ratio of their images by $\tau_n$ and $\tau_n'$, respectively as follows:
    \[
    \tau_n=\frac{\abs{x_n-a_n}}{\abs{x_n-b_n}}\frac{\abs{d_n-b_n}}{\abs{d_n-a_n}}, \ 
\tau_n'=\frac{\abs{f(x_n)-f(a_n)}}{\abs{f(x_n)-f(b_n)}}\frac{\abs{f(d_n)-f(b_n)}}{\abs{f(d_n)-f(a_n)}}.
    \]
    Note that in the case $d_n=\infty$, the cross-ratios reduce to
     \[
    \tau_n=\frac{\abs{x_n-a_n}}{\abs{x_n-b_n}}, \ 
\tau_n'=\frac{\abs{f(x_n)-f(a_n)}}{\abs{f(x_n)-f(b_n)}}.
    \]

    Now assume $\tau_n\rightarrow\tau$ with $0<\tau<\infty$. Note that in order to prove $\tau_n'\rightarrow\tau$ we only need to show that any subsequence of $\{\tau_n'\}$ has a further subsequence that converges to $\tau$. Thus in the following argument we will freely pass to subsequences as needed in order for the limits involved to exist. We start with assuming $\tau_n'\rightarrow \tau'$ and aim to show that 
    \[\tau'=\tau.\]

    Next, let $E_n$ and $F_n$ denote the disjoint subarcs of $C_n$ joining $x_n$ to $ b_n$ and $a_n$ to $d_n$, respectively (see Figure 1 below). 
\begin{figure}[h]

\centering   

\begin{tikzpicture}[x=0.75pt,y=0.75pt,yscale=-.8,xscale=.8]

\draw   (228.77,132.1) .. controls (228.77,89.22) and (265.25,54.46) .. (310.25,54.46) .. controls (355.25,54.46) and (391.72,89.22) .. (391.72,132.1) .. controls (391.72,174.97) and (355.25,209.73) .. (310.25,209.73) .. controls (265.25,209.73) and (228.77,174.97) .. (228.77,132.1) -- cycle ;
\draw  [color={rgb, 255:red, 208; green, 2; blue, 27 }  ,draw opacity=1 ][fill={rgb, 255:red, 208; green, 2; blue, 27 }  ,fill opacity=1 ] (258.1,195.6) .. controls (258.1,192.84) and (260.09,190.6) .. (262.55,190.6) .. controls (265.01,190.6) and (267,192.84) .. (267,195.6) .. controls (267,198.36) and (265.01,200.6) .. (262.55,200.6) .. controls (260.09,200.6) and (258.1,198.36) .. (258.1,195.6) -- cycle ;
\draw  [color={rgb, 255:red, 208; green, 2; blue, 27 }  ,draw opacity=1 ][fill={rgb, 255:red, 208; green, 2; blue, 27 }  ,fill opacity=1 ] (230.1,157.6) .. controls (230.1,154.84) and (232.09,152.6) .. (234.55,152.6) .. controls (237.01,152.6) and (239,154.84) .. (239,157.6) .. controls (239,160.36) and (237.01,162.6) .. (234.55,162.6) .. controls (232.09,162.6) and (230.1,160.36) .. (230.1,157.6) -- cycle ;
\draw  [color={rgb, 255:red, 208; green, 2; blue, 27 }  ,draw opacity=1 ][fill={rgb, 255:red, 208; green, 2; blue, 27 }  ,fill opacity=1 ] (240.1,86.6) .. controls (240.1,83.84) and (242.09,81.6) .. (244.55,81.6) .. controls (247.01,81.6) and (249,83.84) .. (249,86.6) .. controls (249,89.36) and (247.01,91.6) .. (244.55,91.6) .. controls (242.09,91.6) and (240.1,89.36) .. (240.1,86.6) -- cycle ;
\draw  [color={rgb, 255:red, 208; green, 2; blue, 27 }  ,draw opacity=1 ][fill={rgb, 255:red, 208; green, 2; blue, 27 }  ,fill opacity=1 ] (379.1,165.6) .. controls (379.1,162.84) and (381.09,160.6) .. (383.55,160.6) .. controls (386.01,160.6) and (388,162.84) .. (388,165.6) .. controls (388,168.36) and (386.01,170.6) .. (383.55,170.6) .. controls (381.09,170.6) and (379.1,168.36) .. (379.1,165.6) -- cycle ;
\draw [color={rgb, 255:red, 208; green, 2; blue, 27 }  ,draw opacity=1 ][line width=2.25]    (234.55,157.6) .. controls (240,175.6) and (214,126.6) .. (244.55,86.6) ;
\draw [color={rgb, 255:red, 208; green, 2; blue, 27 }  ,draw opacity=1 ][line width=2.25]    (262.55,195.6) .. controls (274,208.6) and (357,226.6) .. (383.55,165.6) ;
\draw [color={rgb, 255:red, 80; green, 227; blue, 194 }  ,draw opacity=1 ]   (232,111.6) .. controls (272,81.6) and (302,232.6) .. (342,202.6) ;
\draw [color={rgb, 255:red, 80; green, 227; blue, 194 }  ,draw opacity=1 ]   (228.77,132.1) .. controls (268.77,102.1) and (326,220.6) .. (366,190.6) ;
\draw [color={rgb, 255:red, 80; green, 227; blue, 194 }  ,draw opacity=1 ]   (231,121.6) .. controls (218.08,131.31) and (162,240.6) .. (287,204.6) ;

\draw (232.55,154.2) node [anchor=south east] [inner sep=0.75pt]  [font=\small]  {$x_{n}$};
\draw (262.57,206.35) node [anchor=north] [inner sep=0.75pt]  [font=\small]  {$a_{n}$};
\draw (242.55,83.2) node [anchor=south east] [inner sep=0.75pt]  [font=\small]  {$b_{n}$};
\draw (385.55,169) node [anchor=north west][inner sep=0.75pt]  [font=\small]  {$d_{n}$};
\draw (202,101.4) node [anchor=north west][inner sep=0.75pt]    {$E_{n}$};
\draw (342,226.99) node [anchor=south west] [inner sep=0.75pt]    {$F_{n}$};

\end{tikzpicture}
\caption{The Teichm\"uller ring domain in the circle $C_n$.}
\label{fig:lemma}

\end{figure}
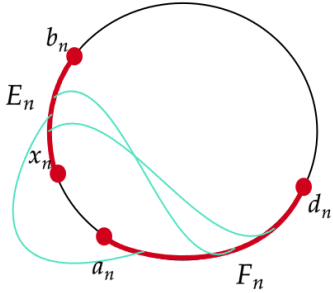
    Then the doubly connected domain $\bar\C\backslash (E_n\cup F_n)$ is M\"obius equivalent to the Teichm\"uller ring domain $R_T(\tau_n)$. Thus, by the conformal invariance of modulus,
    \[
    M(\DeltaEFn)=\Psi(\tau_n).
    \]
    Moreover, as a consequence of the extremal property of Teichm\"uller ring domain, we have 
    \[M(\DeltaEFnprime)\geq \Psi(\tau_n^\prime)\]
    for all $n$. \par
    Our goal is to find an upper bound for $M(\DeltaEFnprime)$ in terms of $M(\DeltaEFn).$ In order to achieve this we proceed as follows. \par
    Recall that $x\in \Ss$ is the common limit point of $x_n,a_n,b_n$. 
    Let 
    \[r_n=\sup \{\abs{f(z)-f(x)}\colon z\in E_n\}\ \text{and} \ R_n=\sqrt{r_n}.\]
    By continuity of $f$ and the fact that $E_n\subset C_n(x_n)$, we have $r_n\to 0$. Thus, since $f$ is asymptotically conformal on $\Ss$, for any $\epsilon>0$ there exists $\delta>0$ such that $f$ is $(1+\epsilon)$-QC in the $\delta$-neighborhood of $\Ss$: $\{z: 1-\delta<|z|<1+\delta\}$. In particular, 
    there exists an $N$ such that $f$ is $(1+\epsilon)$-QC in $f^{-1}(B(f(x),R_n))\subset\{z: 1-\delta<|z|<1+\delta\} $ for all $n\geq N$ as $x\in\Ss$ and $R_n\rightarrow 0$.
    Next, we decompose the curve family $\Gamma_n'=\DeltaEFnprime$ into two subfamilies:
    \[\Gamma_{n,1}^\prime=\{\gamma\in\DeltaEFnprime\colon \gamma\subset B(f(x),R_n)\}, \ \Gamma_{n,2}'=\Gamma_n'-\Gamma_{n,1}'.\]
    Note that each curve contained in $\Gamma_{n,2}'$ joins the circles $\Ss(f(x),r_n)$ and $\Ss(f(x),R_n)$. Then by the monotonicity of the modulus and the majorization principle we have the following chain of inequalities:
    \begin{align*}
        \Psi(\tau_n^{\prime})\leq M(\DeltaEFnprime)&\leq M(\Gamma_{n,1}^\prime)+M(\Gamma_{n,2}^\prime)\\
        &\leq(1+\epsilon)M(\Gamma_{n,1})+\frac{2\pi}{\log\left(\tfrac{R_n}{r_n}\right)}\\
        &\leq(1+\epsilon)M(\Gamma_{n})+\frac{2\pi}{\log\left(\tfrac{R_n}{r_n}\right)}\\
        &= (1+\epsilon)\Psi(\tau_n)+\frac{2\pi}{\log\left(\tfrac{1}{\sqrt{r_n}}\right)}.
    \end{align*}
    By letting $n\rightarrow\infty$ and then letting $\epsilon\rightarrow 0$, one can derive that 
    \[
    \Psi(\tau^\prime)\leq \Psi(\tau).
    \]
    Since the Teichm\"uller funciton is strictly decreasing, we have 
    \[
    \tau^\prime\geq\tau.
    \]
    
    By considering the conjugate configuration, the reverse inequality
    $\tau^\prime\leq \tau$ follows. More precisely, let $E_n$ and $F_n$ denote the disjoint subarcs of $C_n$ joining $x_n$ to $ a_n$ and $b_n$ to $d_n$, respectively. Then we have 
    \[
     M(\DeltaEFn)=\Psi(\frac{1}{\tau_n}), \ M(\DeltaEFnprime)\geq \Psi(\frac{1}{\tau_n^\prime}).
    \]
    Thus the above argument shows that
    \[
    \Psi(\frac{1}{\tau^\prime})\leq \Psi(\frac{1}{\tau}) \ \Rightarrow \ \frac{1}{\tau'}\geq\frac{1}{\tau},
    \]
    and the desired equality $\tau'=\tau$ follows.
\end{proof}

\section{Proof of Theorem 1}\label{Section:3}
Instead of diving headfirst into the proof, we first make some reductions that simplify the presentation and enhance readability. Thus the proof of the main theorem will be divided into $3$ subsections. The first subsection discusses the infinitesimal behavior of $AS$ maps. In the second subsection we will address the proof of Theorem $1$ (a)$\Ra$(b). The proof is by no means trivial, hence this subsection is the kernel of this article. In the final subsection we will discuss the boundary behavior of $AS$ mappings. This will be precisely the last missing part of Theorem $1$. Indeed, the subsection on the boundary correspondence of $AS$ embeddings will yield Theorem $1$ (b)$\Ra$(c). The chain of implications is thus complete, as (c)$\Ra$(a) can be found in \cite[Theorem 3.2]{brania2004asymptotically}.
\subsection{Infinitesimal behavior of $AS$ maps}
Towards the proof of Theorem 1, we first establish an equivalent description for the AS condition (1.4) by using the language of convergent sequences. This formulation plays a central role in our approach and may be of independent interest for the study of AS maps in general. 
\begin{pro}\label{proposition:sequences_AS}
    Let $f$ be an embedding of the unit disk into the complex plane $\C$. Then $f$ is an $AS$ embedding if and only if, for all sequences $x_n,a_n,b_n\in\D$ that converge to the same limit point $x\in\Dc$, the following is true: 
    \begin{equation}
    \lim_{n\rightarrow\infty}\frac{\abs{x_n-a_n}}{\abs{x_n-b_n}}= t \implies \ \lim_{n\rightarrow\infty} \frac{\abs{f(x_n)-f(a_n)}}{\abs{f(x_n)-f(b_n)}}=t
    \end{equation}
for any $t\in [0,\infty]$.
\end{pro}
\begin{proof}
We first deal with the `if' part.
    Suppose, for the sake of contradiction, that $f$ is not $AS$. Then there exist some $\epsilon>0$ and $t>0$ such that for each $\delta_n=1/n$ there exist $x_n,a_n,b_n\in\D$, contained in a $\delta_n$-ball with
    \[
    \frac{\abs{x_n-a_n}}{\abs{x_n-b_n}}\leq t \ \text{and\  } \\\frac{\abs{f(x_n)-f(a_n)}}{\abs{f(x_n)-f(b_n)}}> (1+\epsilon)t.
    \]
    Then, by passing to subsequences if needed, we can assume that $x_n,a_n,b_n$ converge to a common limit point $x\in\bar\D$ and that 
    \[
    \frac{\abs{x_n-a_n}}{\abs{x_n-b_n}}\to t_{1}\leq t.
    \]
    However,
    \[
    \frac{\abs{f(x_n)-f(a_n)}}{\abs{f(x_n)-f(b_n)}}> (1+\epsilon)t> t\geq t_{1}
    \]
  for all $n\geq 1$. This contradicts our assumption (3.1). Hence $f$ must be AS.

\vspace{0.2cm}
Next we deal with the "only if" part.  
Suppose that $f$ is an $AS$ embedding and fix three sequences $x_n,a_n,b_n\in\D$ converging to the same limit point $x\in\bar\D$, with $\frac{\abs{x_n-a_n}}{\abs{x_n-b_n}}\to t\in [0,\infty]$ as n goes to infinity. We need to show that 
\begin{equation}
\lim_{n\rightarrow\infty}\frac{\abs{f(x_n)-f(a_n)}}{\abs{f(x_n)-f(b_n)}}=t.
\end{equation}
To this end, we first assume that $t=0$. Then, for any fixed $\epsilon_1>0$, there exists integer $N_1$ such that 
$$n\geq N_1 \ \Rightarrow \ \frac{\abs{x_n-a_n}}{\abs{x_n-b_n}}<\epsilon_1. $$
Furthermore, by the AS condition (1.4) and the fact that $x_n,a_n,b_n\rightarrow x$, for any $\epsilon>0$ there exists integer $N\geq N_1$ such that 
$$n\geq N \ \Rightarrow \ \frac{\abs{x_n-a_n}}{\abs{x_n-b_n}}<\epsilon_1 \ 
 \text{and} \ \frac{\abs{f(x_n)-f(a_n)}}{\abs{f(x_n)-f(b_n)}}\leq (1+\epsilon)\epsilon_1. $$
 This verifies (3.2) for the case $t=0$. By considering the reciprocal ratio (switching the role of $a_n$ and $b_n$), the case $t=\infty$ reduces to the case $t=0$. 

 For the case $0<t<\infty$, using the AS condition (1.4) again, one can deduce that for any $\epsilon_1,\epsilon>0$ there exists integer $N$ such that for all $n\geq N$, we have 
\begin{equation}\label{two_sided_AS}
t-\epsilon_1\leq \frac{\abs{x_n-a_n}}{\abs{x_n-b_n}}\leq t+\epsilon_1 \
\text{and} \ \frac{t-\epsilon_1}{1+\epsilon}\leq \frac{\abs{f(x_n)-f(a_n)}}{\abs{f(x_n)-f(b_n)}}\leq (1+\epsilon)(t+\epsilon_1). 
\end{equation}
 Since $\epsilon_1, \epsilon>0$ are arbitrary, (3.2) follows as desired.
\end{proof}

\subsection{Proof Theorem 1 (a)$\Ra$ (b)}
We are now fully equipped to deal with the proof of Theorem \ref{thm:main_theorem}. Recall, that in order to establish the validity of Theorem 1, it remains to show that (a) $\Rightarrow$ (b) $\Rightarrow$ (c). We formulate these two implications in Theorems 2 and 3, respectively. 
\begin{thm}\label{thm:a_ra_b}
    A conformal map $f$ of the unit disk $\D$ onto a symmetric quasidisk $G$ is an $AS$ embedding.
\end{thm}
The proof of Theorem 2 is the main part of this article. It utilizes the analytic and geometric tools discussed in Section 2 and the description of $AS$ embeddings in the language of convergent sequences which was given in Proposition \ref{proposition:sequences_AS}. We will divide the proof of Theorem \ref{thm:a_ra_b} into subsections.

\subsubsection{Reduction and notation}
Let $f: \D\rightarrow G$ be a conformal map as in Theorem 2. Since $G$ is a symmetric quasidisk, by Lemma 1 $f$ has a quasiconformal extension to $\C$ that is asymptotically conformal on the unit circle. For the entire proof, we should take $f$ as such an extension. For simplicity of notation, the image under $f$ will be denoted by the `prime' notation: $p'=f(p)$ for a point or set $p$. 

In order to use Proposition \ref{proposition:sequences_AS} to show that $f$ is AS in $\D$, we fix sequences $x_n,a_n,b_n\in\D$, that converge to a common limit point $x\in\bar\D$ with
\begin{equation}
 \lim_{n\rightarrow\infty}\frac{\abs{x_n-a_n}}{\abs{x_n-b_n}}=t\in [0,\infty].
 \end{equation}
 According to Proposition \ref{proposition:sequences_AS}, we need to show that 
\begin{equation}
 \lim_{n\rightarrow\infty}\frac{\abs{x_n'-a_n'}}{\abs{x_n'-b_n'}}=t.
 \end{equation}
 Denote the ratios in (3.4) and (3.5) by $t_n$ and $t_n'$, respectively. Since the quasiconformal extension $f$ is also QS, if $t=0$ or $\infty$, (3.5) follows from (3.4) immediately by the QS property. 

 Another reduction we can make is that, if the common limit point $x$ is inside the unit disk, then (3.5) follows from the Cauchy integral formula for analytic functions. In fact, applying the Cauchy integral formula to $f(z)$ on a small fixed circle $|z-x|=r$, one can deduce that 
 $$\frac{x_n'-a_n'}{x_n-a_n}=\frac{1}{2\pi i}\int_{|z-x|=r}\cfrac{f(z)dz}{(z-x_n)(z-a_n)}.
 $$
 Thus it follows that 
 $$\frac{t_n'}{t_n}=\frac{|x_n'-a_n'|}{|x_n-a_n|}\cdot \frac{|x_n-b_n|}{|x_n'-b_n'|}\rightarrow 1. 
 $$

After these reductions, for the remainder of the proof, we assume that $x_n,a_n,b_n\rightarrow x\in \Ss$ and that $t_n\rightarrow t\in (0,\infty)$. Furthermore, in order to show that $t_n'\rightarrow t$, it suffices to show that each subsequence of $t_n'$ has a further subsequence that converges to $t$. Thus, in the argument below we will pass to subsequences freely as needed and still keep the original notation for the sequences involved. 

\subsubsection{Separated configuration}
For each $n$, let $C_n$ denote the unique circle (or straight line) that passes through the three distinct points $x_n,a_n,b_n$, $D_n$ the disk bounded by $C_n$, and $r_n$ the radius of $C_n$. We say that $a_n$ and $b_n$ are {\it separated by $x_n$} (hence the {\it separated configuration}), if they are on different semicircles of $C_n$ that are cut out by the diameter of $D_n$ through $x_n$. If the points are collinear, then we say that $a_n,b_n$ are separated by $x_n$ if $x_n\in[a_n,b_n]$. 
We shall treat the separated configuration in this subsection. The non-separated configuration will be treated in the next subsection by using a reflection argument combined with analytic properties of the conformal map $f$. By passing to subsequences if needed, we can divide the argument into the following three cases.

\vspace{0.3cm}
\begin{case}{1.1} $x_n,a_n,b_n$ are collinear (see Figure 2). In this case, by choosing the fourth point $d_n$ as $\infty$ in Lemma \ref{lem:Modulus_estimates}, one immediately derives that 
\[
\lim_{n\to\infty}\frac{\abs{x_n-a_n}}{\abs{x_n-b_n}}=t \ \implies \ \lim_{n\to\infty}\frac{\abs{\xnp-\anp}}{\abs{\xnp-\bnp}}=t.
\] 
\begin{figure}[h]
    \centering
    \includegraphics[width=1\linewidth]{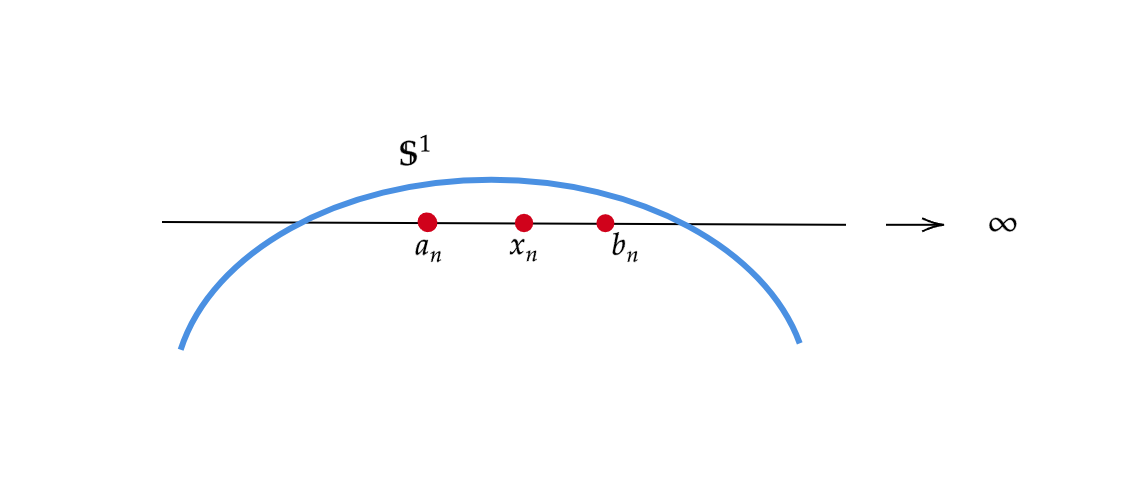}
    \caption{The collinear case in separated configuration}
    \label{fig:enter-label_1}
\end{figure}
\end{case}

\begin{case}{1.2} $x_n,a_n,b_n$ are not collinear and 
\begin{equation}\lim_{n\rightarrow\infty}\frac{\abs{x_n-a_n}}{r_n}= 0.
\end{equation}
In this case we shall apply Lemma 2 by letting $d_n$ be the point antipodal to $x_n$ on the circle  $C_n$. Observe that the point $d_n$ can potentially lie outside the unit disk. As in Lemma 2, let $\tau_n$ and $\tau_n'$ denote the corresponding cross-ratios:
 $$\tau_n=\frac{\abs{x_n-a_n}}{\abs{x_n-b_n}}\frac{\abs{d_n-b_n}}{\abs{d_n-a_n}}, \ \tau_n^{\prime}=\frac{\abs{\xnp-\anp}}{\abs{\xnp-\bnp}}\frac{\abs{\dnp-\bnp}}{\abs{\dnp-\anp}}.$$

We will use the limits of $\tau_n$ and $\tau_n'$ to derive the limit of $t_n'$ (hence (3.5)). We start with a few simple observations on how the points are relatively located. Since $\abs{x_n-d_n}=2r_n$, it follows from (3.6) that 
\begin{equation}\label{eq3.1}
    \lim_{n\rightarrow\infty}\frac{\abs{x_n-a_n}}{\abs{x_n-d_n}}=\lim_{n\rightarrow\infty}\frac{\abs{x_n-b_n}}{\abs{x_n-d_n}}=0.
\end{equation}
\begin{figure}[h]
    \centering
    \includegraphics[width=1\linewidth]{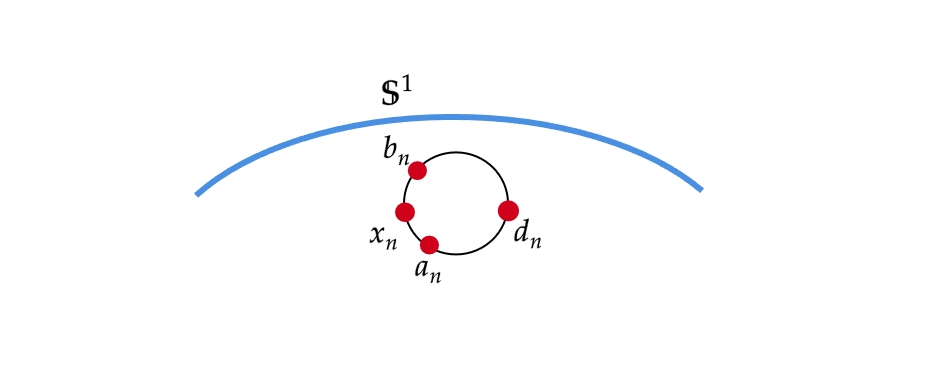}
    \caption{The point $d_n$ is a relative "infinity" to the points $x_n,a_n,b_n$.}
    \label{fig:enter-label_2}
\end{figure}\\
Due to (3.7), we have that $\abs{d_n-x_n}\geq \abs{x_n-a_n}$ for $n$ sufficiently large, and thus the following double-sided triangle inequality holds:
\begin{align*}
    \abs{d_n-x_n}-\abs{x_n-a_n}\leq \abs{d_n-a_n}\leq \abs{d_n-x_n}+\abs{x_n-a_n}.
\end{align*}
Moreover, dividing it by $\abs{d_n-x_n}$, and using (3.7) again, we conclude that 
\begin{equation}
    \frac{\abs{d_n-a_n}}{\abs{d_n-x_n}}\to 1.
\end{equation}
In a similar fashion we obtain, 
\begin{equation}
    \frac{\abs{d_n-b_n}}{\abs{d_n-x_n}}\to 1.
\end{equation}

 Next, by quasisymmetry of $f$, (3.7) holds for the image points of $x_n$, $a_n$,$b_n,d_n$: 
\begin{equation}\label{eq3.2}
       \lim_{n\rightarrow\infty}\frac{\abs{x_n'-a_n'}}{\abs{x_n'-d_n'}}=\lim_{n\rightarrow\infty}\frac{\abs{x_n'-b_n'}}{\abs{x_n'-d_n'}}=0.
\end{equation}
Thus, similar to (3.8) and (3.9), we obtain that
\begin{equation}
    \frac{\abs{\dnp-\anp}}{\abs{\dnp-\xnp}}\to 1\quad\text{and}\quad \frac{\abs{\dnp-\bnp}}{\abs{\dnp-\xnp}}\to 1.
\end{equation}

Finally, it follows from (3.8) and (3.9) that
$$\lim_{n\rightarrow\infty}\tau_n=\lim_{n\rightarrow\infty}t_n=t.$$
Therefore, from (3.11) and Lemma 2 one concludes that
$$\lim_{n\rightarrow\infty}t_n'=\lim_{n\rightarrow\infty}\tau_n'=\lim_{n\rightarrow\infty}\tau_n=t.$$
\end{case}
\vspace{0.3cm}
\begin{case}{1.3} $x_n,a_n,b_n$ are not collinear and 
\begin{equation}\lim_{n\rightarrow\infty}\frac{\abs{x_n-a_n}}{r_n}= r>0.
\end{equation}
The essential difference between this case and the above case is that, in the above case, (3.6) implies that the radius $r_n$ is much bigger than the distance $|x_n-a_n|$ and thus one could choose a fourth point $d_n$ on $C_n$ that is relatively far away from $x_n,a_n,b_n$, acting as the role of the point at infinity as in the collinear case. Therefore, the estimates on the four point cross-ratios $\tau_n$ and $\tau_n'$ can be transferred to estimates on the three point ratios $t_n$ and $t_n'$ as desired. This approach alone does not work in the current case. We need to bring in another tool, namely the behavior of the derivative of $f$ as stated in Lemma \ref{lemma:pommerenke} statement (2).

This is the most complicated case. In order to make the argument easier to follow, we further divide it into three subcases. However, before proceeding, we want to point out that the arguments in this case do not depend on the separation property described above.

\vspace{0.3cm}
\begin{case}{1.3.1} $x_n$ is relatively far away from the boundary of the unit disk so that the uniform convergence in Lemma \ref{lemma:pommerenke} can be applied. More precisely, assume that there is a constant $\lambda>0$ such that (for all large $n$) 
\begin{equation}
    \frac{1-|x_n|}{r_n}\geq\lambda.
\end{equation}
Then it follows that
\[
    \frac{\abs{x_n-a_n}}{1-\abs{x_n}}\leq\frac{2r_n}{1-|x_n|}\leq \frac{2}{\lambda}, \quad \frac{\abs{x_n-b_n}}{1-\abs{x_n}}\leq \frac{2}{\lambda}.
\]
Hence by the uniform convergence condition in Lemma \ref{lemma:pommerenke} (2),  one concludes that
\begin{align}\label{application:pomm2}
\frac{f(x_n)-f(a_n)}{(x_n-a_n)f'(x_n)}\to 1 \quad\text{and}\quad \frac{f(x_n)-f(b_n)}{(x_n-b_n)f'(x_n)}\to 1 
\end{align}
as $n\to\infty$. Furthermore,
\[
\frac{t_n'}{t_n}=\frac{\abs{f(x_n)-f(a_n)}}{\abs{x_n-a_n}\abs{f'(x_n)}}\cdot\frac{\abs{x_n-b_n}\abs{f'(x_n)}}{\abs{f(x_n)-f(b_n)}}.
\]
This combined with (\ref{application:pomm2}) and (3.4) (with $0<t<\infty$) gives the following limits as desired:
\begin{align}
    \lim_{n\to\infty}t_n'=\lim_{n\to\infty}t_n=t.
\end{align}
\end{case}

\vspace{0.3cm}
\begin{case}{1.3.2} Next, we consider the case when $a_n$ and $b_n$ are relatively close to one another:
\begin{equation}
    \lim_{n\rightarrow\infty}\frac{|a_n-b_n| }{r_n} =0.
\end{equation}
This case can be dealt with by a simple QS argument, similar to the one used to establish (3.11) above. In fact, it follows from (3.4), (3.12), and (3.16) that 
$$\lim_{n\rightarrow\infty}\frac{|a_n-b_n|}{|x_n-b_n|}=\lim_{n\rightarrow\infty}\frac{|a_n-b_n|}{r_n}\cdot
\frac{r_n}{|x_n-a_n|}\cdot\frac{|x_n-a_n|}{|x_n-b_n|}=0\cdot\frac{1}{r}\cdot t=0.$$
Thus, by quasisymmetry of $f$,
$$\lim_{n\rightarrow\infty}\frac{|a_n'-b_n'|}{|x_n'-b_n'|}=0.$$
Therefore, by routine application of triangle inequalities, one derives that 
$$\lim_{n\rightarrow\infty}\frac{|x_n-a_n|}{|x_n-b_n|}=1 \ \text{and} \ 
\lim_{n\rightarrow\infty}\frac{|x_n'-a_n'|}{|x_n'-b_n'|}=1.$$
\end{case}

\vspace{0.3cm}
\begin{case}{1.3.3} Finally, we derive (3.5) under the assumption that neither (3.13) nor (3.16) holds. By passing to subsequences again if needed, we may further assume that 
\begin{equation}   
\lim_{n\rightarrow\infty}\frac{1-|x_n|}{r_n}=0 \ \text{and} \ 
\lim_{n\rightarrow\infty}\frac{|a_n-b_n| }{r_n} =s>0.
\end{equation}
Towards the goal of deriving (3.5) from these assumptions, we shall construct a fourth point $d_n\in C_n\backslash C_n(x_n)$ as in Lemma 2 such that 
\begin{equation}
\frac{1-|d_n|}{r_n}\geq\lambda, \ \frac{|d_n-a_n|}{|d_n-b_n|}=1
\end{equation}
for some constant $\lambda>0$. To keep the flow of ideas, we postpone the construction of $d_n$ and proceed with such a fourth point being given. 

A direct application of Lemma 2, together with the second part of (3.18), yields that 
$$\lim_{n\rightarrow\infty}\frac{\abs{\xnp-\anp}}{\abs{\xnp-\bnp}}\frac{\abs{\dnp-\bnp}}{\abs{\dnp-\anp}}=\lim_{n\rightarrow\infty}\tau_n=t.
$$
Furthermore, replacing $x_n$ by $d_n$ in the argument for (3.14) and (3.15), using Lemma 1, we conclude that 
$$\lim_{n\rightarrow\infty}\frac{\abs{\dnp-\bnp}}{\abs{\dnp-\anp}}=\lim_{n\rightarrow\infty}\frac{\abs{d_n-b_n}}{\abs{d_n-a_n}}=1.
$$
Combining the above two yields (3.5) as desired.\par
We now turn to the construction of the point $d_n$ that satisfies (3.18) to complete the proof. Under the standing assumptions (3.12) and (3.17), we claim that the mid point $d_n$ of the arc $C(a_n,b_n)$ joining $a_n$ and $b_n$ on $C_n$ that does not contain $x_n$ will satisfy (3.18). In fact, the equality in (3.18) just follows from the definition of $d_n$. The inequality in (3.18) is not difficult to see from geometrical intuition. But it is rather technical to construct a rigorous proof as one can see below. First we rewrite the inequality in (3.18) in the limit form:
\begin{equation}
    \lim_{n\rightarrow\infty}\frac{\delta(d_n)}{r_n}=\lambda>0,
\end{equation}
where $\delta(\cdot)$ denotes the distance of a point to the unit circle and is for the convenience of argument to follow. 

Note that, since $r_n\rightarrow 0$ by (3.12), all the points $x_n,a_n,b_n,d_n$ accumulate at the common limit point $x\in\Ss$. Thus in estimating and comparing relevant distances near $x$, one can regard $\Ss$ as a line. More precisely, one can choose a M\"obius transformation $\varphi$ that maps the unit circle to the real line with $\varphi(x)=0$, $\varphi(-x)=\infty$, and $|\varphi'(x)|=1$. It follows that
$$\lim_{a,b\rightarrow x}\frac{|\varphi(a)-\varphi(b)|}{|a-b|}=|\varphi'(x)|=1.$$
Therefore, to verify (3.19) one can replace $\Ss$ by the real line $\R$.

To this goal, we first assume that $C_n$ is contained in the upper half plane. Denote by $w_n$ the south pole of $C_n$ (or the point on $C_n$ closest to $\R$). Since $\delta(x_n)/r_n\rightarrow 0$ by (3.17), may assume that $x_n$ is located in the lower left quarter of $C_n$. 
Let the upper case letter $X_n$ denote the angle subtended by the smaller arc from reference point $w_n$ to $x_n$ (and similarly for other points $a_n,b_n,d_n$ involved in the argument). 

With the help of Figure \ref{fig:C_n_not_in_D}, it follows from elementary geometry and trigonometry that 
$$\delta(x_n)\geq r_n-r_n\cos X_n, \ D_n>\min\{A_n,B_n\},
$$ and 
$$A_n=2\sin^{-1}\left(\frac{|x_n-a_n|}{2r_n} \right)\pm X_n, \
B_n=2\sin^{-1}\left(\frac{|x_n-b_n|}{2r_n} \right)\pm X_n,
$$
where the the sign $\pm$ depends on the relative positions of points involved. Letting $n\rightarrow\infty$, by (3.12), (3.17), and (3.4) one deduces that 
$$X_n\rightarrow 0, \ A_n\rightarrow A=2\sin^{-1}\left(\frac{r}{2}\right), \ B_n\rightarrow B=2\sin^{-1}\left(\frac{r}{2t}\right), \ \text{and}$$
$$\lim_{n\rightarrow\infty}D_n=D\geq\min\{A,B\}.$$
A simple trigonometry argument again shows that $\delta(d_n)\geq r_n-r_n\cos D_n$. Thus it follows that 
 $$\lim_{n\rightarrow\infty}\frac{\delta(d_n)}{r_n}=\lambda\geq 1-\cos D>0$$
 as desired.
\begin{figure}[h]
    \centering
    \includegraphics[width=1\linewidth]{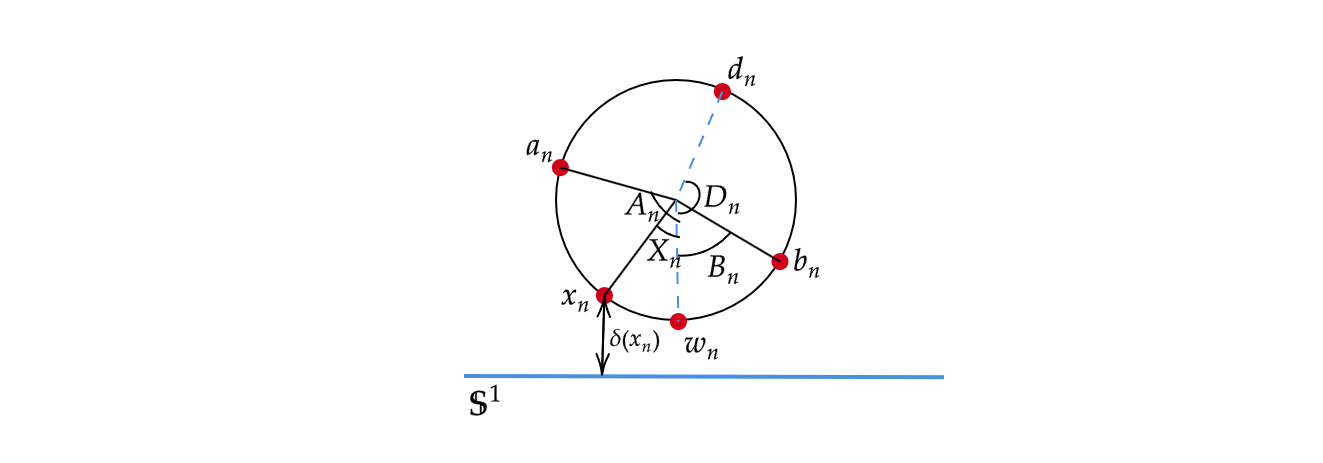}
    \caption{Construction of $d_n$ in Case 1.3.3 when $C_n$ is contained in $\D$.}
    \label{fig:C_n_not_in_D}
\end{figure}

Next, we consider the case when $C_n$ is not entirely contained in the upper half plane. It remains to show that, in this case, the middle point $d_n$ on the arc $C(a_n,b_n)$ also satisfies the distance condition (3.19). To this end, we first establish a claim. 

\vspace{0.3cm}
\textbf{Claim 1.} Let $w_n$ be the point on $C_n\cap\R$ that is closest to $x_n$. If $a_n$ or $b_n$ is contained in the shorter arc $\arc{x_nw_n}$ and (3.12) holds, then there exists a $\lambda_1>0$ such that $\frac{\delta(x_n)}{r_n}\geq \lambda_1$. (See Figure \ref{fig:case3_construction} for reference.)
\begin{proof}[Proof of Claim 1.]
Suppose without loss of generality that $b_n$ is contained in the arc connecting $x_n$ to $w_n$. 
Let $L$ denote the line through $b_n$ parallel to the real line $\R$, $\hat{b}_n$ the other intersection point of $L$ and $C_n$, and $\hat{x}_n$ the perpendicular point from $x_n$ to $L$ (See figure \ref{fig:case3_construction}).  
Consider the triangle with vertices at $x_n,b_n,\hat{x}_n$, and denote by $\beta_n$ the angle at $b_n$ of this triangle. 
Then a simple geometric observation based on the corresponding figure yields that
\begin{equation}
\beta_n=\frac{1}{2}\frac{l(\arc{x_n\hat{b}_n})} {r_n}\geq\frac{1}{2}\frac{|x_n-b_n|}{r_n}.
\end{equation}
This, together with (3.4) and (3.12), yields that 
$$\begin{aligned}\frac{\delta(x_n)}{r_n}\geq\frac{\abs{x_n-\hat{x}_n}}{r_n}&=\frac{\sin(\beta_n)\abs{x_n-b_n}}{r_n}\\
&\geq\sin\left(\frac{1}{2}\frac{|x_n-b_n|}{r_n} \right)\frac{|x_n-b_n|}{r_n}
\rightarrow\sin\left(\frac{1}{2}\frac{r}{t} \right)\frac{r}{t}>0.
\end{aligned}$$
This completes the proof of Claim 1.
       
\begin{figure}[h]

\begin{tikzpicture}[x=0.75pt,y=0.75pt,yscale=-.8,xscale=.8]

\draw [color={rgb, 255:red, 74; green, 144; blue, 226 }  ,draw opacity=1 ][line width=1.5]    (25,160) -- (265,160.6) ;
\draw   (71,143) .. controls (71,109.86) and (97.86,83) .. (131,83) .. controls (164.14,83) and (191,109.86) .. (191,143) .. controls (191,176.14) and (164.14,203) .. (131,203) .. controls (97.86,203) and (71,176.14) .. (71,143) -- cycle ;
\draw  [color={rgb, 255:red, 208; green, 2; blue, 27 }  ,draw opacity=1 ][fill={rgb, 255:red, 208; green, 2; blue, 27 }  ,fill opacity=1 ] (67.4,136.8) .. controls (67.4,134.7) and (69.1,133) .. (71.2,133) .. controls (73.3,133) and (75,134.7) .. (75,136.8) .. controls (75,138.9) and (73.3,140.6) .. (71.2,140.6) .. controls (69.1,140.6) and (67.4,138.9) .. (67.4,136.8) -- cycle ;
\draw  [color={rgb, 255:red, 208; green, 2; blue, 27 }  ,draw opacity=1 ][fill={rgb, 255:red, 208; green, 2; blue, 27 }  ,fill opacity=1 ] (158.4,92.8) .. controls (158.4,90.7) and (160.1,89) .. (162.2,89) .. controls (164.3,89) and (166,90.7) .. (166,92.8) .. controls (166,94.9) and (164.3,96.6) .. (162.2,96.6) .. controls (160.1,96.6) and (158.4,94.9) .. (158.4,92.8) -- cycle ;
\draw  [color={rgb, 255:red, 208; green, 2; blue, 27 }  ,draw opacity=1 ][fill={rgb, 255:red, 208; green, 2; blue, 27 }  ,fill opacity=1 ] (184.4,125.8) .. controls (184.4,123.7) and (186.1,122) .. (188.2,122) .. controls (190.3,122) and (192,123.7) .. (192,125.8) .. controls (192,127.9) and (190.3,129.6) .. (188.2,129.6) .. controls (186.1,129.6) and (184.4,127.9) .. (184.4,125.8) -- cycle ;
\draw  [color={rgb, 255:red, 208; green, 2; blue, 27 }  ,draw opacity=1 ][fill={rgb, 255:red, 208; green, 2; blue, 27 }  ,fill opacity=1 ] (184.4,159.8) .. controls (184.4,157.7) and (186.1,156) .. (188.2,156) .. controls (190.3,156) and (192,157.7) .. (192,159.8) .. controls (192,161.9) and (190.3,163.6) .. (188.2,163.6) .. controls (186.1,163.6) and (184.4,161.9) .. (184.4,159.8) -- cycle ;
\draw [color={rgb, 255:red, 74; green, 144; blue, 226 }  ,draw opacity=1 ][line width=1.5]    (344,161) -- (584,161.6) ;
\draw   (390,144) .. controls (390,110.86) and (416.86,84) .. (450,84) .. controls (483.14,84) and (510,110.86) .. (510,144) .. controls (510,177.14) and (483.14,204) .. (450,204) .. controls (416.86,204) and (390,177.14) .. (390,144) -- cycle ;
\draw  [color={rgb, 255:red, 208; green, 2; blue, 27 }  ,draw opacity=1 ][fill={rgb, 255:red, 208; green, 2; blue, 27 }  ,fill opacity=1 ] (386.4,137.8) .. controls (386.4,135.7) and (388.1,134) .. (390.2,134) .. controls (392.3,134) and (394,135.7) .. (394,137.8) .. controls (394,139.9) and (392.3,141.6) .. (390.2,141.6) .. controls (388.1,141.6) and (386.4,139.9) .. (386.4,137.8) -- cycle ;
\draw  [color={rgb, 255:red, 208; green, 2; blue, 27 }  ,draw opacity=1 ][fill={rgb, 255:red, 208; green, 2; blue, 27 }  ,fill opacity=1 ] (477.4,93.8) .. controls (477.4,91.7) and (479.1,90) .. (481.2,90) .. controls (483.3,90) and (485,91.7) .. (485,93.8) .. controls (485,95.9) and (483.3,97.6) .. (481.2,97.6) .. controls (479.1,97.6) and (477.4,95.9) .. (477.4,93.8) -- cycle ;
\draw  [color={rgb, 255:red, 208; green, 2; blue, 27 }  ,draw opacity=1 ][fill={rgb, 255:red, 208; green, 2; blue, 27 }  ,fill opacity=1 ] (503.4,126.8) .. controls (503.4,124.7) and (505.1,123) .. (507.2,123) .. controls (509.3,123) and (511,124.7) .. (511,126.8) .. controls (511,128.9) and (509.3,130.6) .. (507.2,130.6) .. controls (505.1,130.6) and (503.4,128.9) .. (503.4,126.8) -- cycle ;
\draw  [color={rgb, 255:red, 208; green, 2; blue, 27 }  ,draw opacity=1 ][fill={rgb, 255:red, 208; green, 2; blue, 27 }  ,fill opacity=1 ] (503.4,160.8) .. controls (503.4,158.7) and (505.1,157) .. (507.2,157) .. controls (509.3,157) and (511,158.7) .. (511,160.8) .. controls (511,162.9) and (509.3,164.6) .. (507.2,164.6) .. controls (505.1,164.6) and (503.4,162.9) .. (503.4,160.8) -- cycle ;
\draw [color={rgb, 255:red, 0; green, 0; blue, 0 }  ,draw opacity=1 ][line width=1.5]    (344,125.6) -- (582,126.6) ;
\draw  [color={rgb, 255:red, 208; green, 2; blue, 27 }  ,draw opacity=1 ][fill={rgb, 255:red, 208; green, 2; blue, 27 }  ,fill opacity=1 ] (477.4,126.8) .. controls (477.4,124.7) and (479.1,123) .. (481.2,123) .. controls (483.3,123) and (485,124.7) .. (485,126.8) .. controls (485,128.9) and (483.3,130.6) .. (481.2,130.6) .. controls (479.1,130.6) and (477.4,128.9) .. (477.4,126.8) -- cycle ;
\draw    (481.2,93.8) -- (481.2,126.8) ;
\draw    (481.2,93.8) -- (507.2,126.8) ;

\draw (164.2,89.4) node [anchor=south west] [inner sep=0.75pt]  [font=\small]  {$x_{n}$};
\draw (69.2,140.2) node [anchor=north east] [inner sep=0.75pt]  [font=\small]  {$a_{n}$};
\draw (190.2,122.4) node [anchor=south west] [inner sep=0.75pt]  [font=\small]  {$b_{n}$};
\draw (43,164.4) node [anchor=north west][inner sep=0.75pt]    {$\mathbb{S}^{1}$};
\draw (190.2,163.2) node [anchor=north west][inner sep=0.75pt]  [font=\small]  {$w_{n}$};
\draw (483.2,90.4) node [anchor=south west] [inner sep=0.75pt]  [font=\small]  {$x_{n}$};
\draw (388.2,141.2) node [anchor=north east] [inner sep=0.75pt]  [font=\small]  {$a_{n}$};
\draw (509.2,123.4) node [anchor=south west] [inner sep=0.75pt]  [font=\small]  {$b_{n}$};
\draw (362,165.4) node [anchor=north west][inner sep=0.75pt]    {$\mathbb{S}^{1}$};
\draw (509.2,164.2) node [anchor=north west][inner sep=0.75pt]  [font=\small]  {$w_{n}$};
\draw (483.2,130.2) node [anchor=north west][inner sep=0.75pt]  [font=\small]  {$\hat{x}_{n}$};
\draw (359.23,118) node [anchor=south] [inner sep=0.75pt]    {$L$};

\end{tikzpicture}
\caption{$C_n$ is not entirely contained in $\D$. However, in this case a positive proportion of $C_n$ is always contained in $\D.$}
\label{fig:case3_construction}
\end{figure}
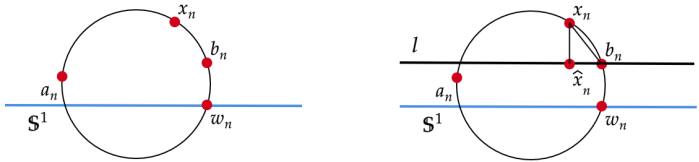

\end{proof}

\par
With Claim 1 in our hands, now we proceed to show that the point $d_n$ constructed above satisfies (3.19). Recall that we find ourselves in the Case 1.3.3, meaning 
that (3.17) holds. In light of the first limit in (3.17) and Claim 1, we see that the arc $C(a_n,b_n)$ is entirely contained in the upper half plane (See Figure \ref{fig:d_n_equidistant}). Thus its middle point $d_n$ is in the upper half plane as well. Furthermore, the second limit in (3.17) allows us to swap $x_n$ for $d_n$ in Claim 1. Thus (3.19) (and hence (3.18)) is satisfied by $d_n$. This completes the proof for Case 1.3.3.  
\begin{figure}[h]
    \centering
    \includegraphics[width=1\linewidth]{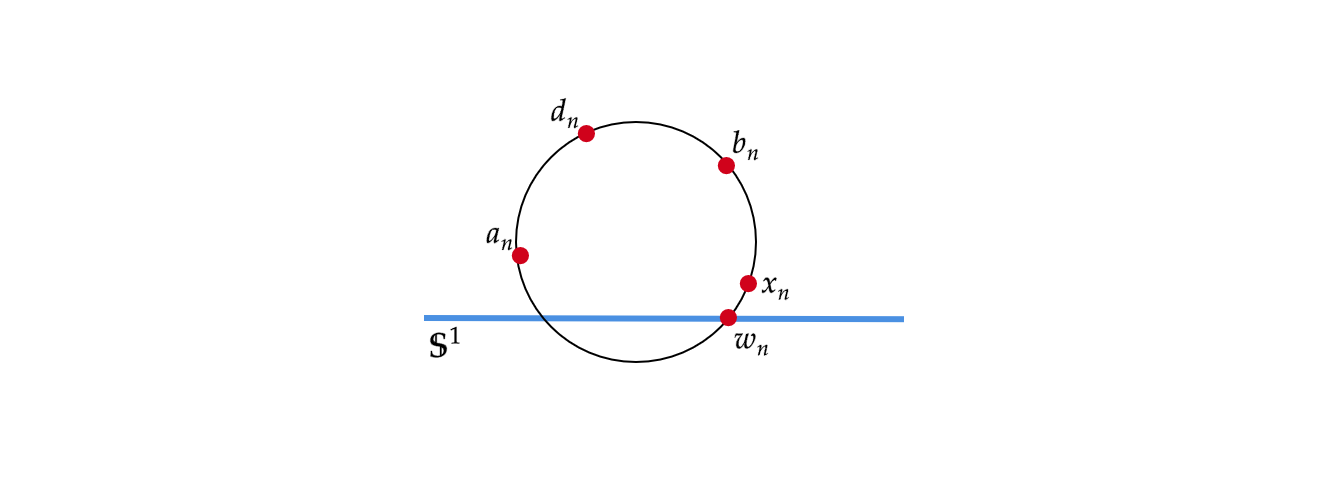}
    \caption{We choose $d_n$ equidistant from $a_n,b_n$}
    \label{fig:d_n_equidistant}
\end{figure}
\end{case}
\end{case}
\vspace{2mm}
\subsubsection{Non-separated configuration}
It still remains to consider the other configuration, when the points $a_n,b_n$ are not separated by $x_n$ in the above sense. To deal with this configuration, we use what we call a reflection method. We will first explain what this reflection operation is and then utilize it to achieve our end goal. Given  $a_n,b_n,x_n$ as in subsection 3.2.1 satisfying (3.4) such that $a_n,b_n$ are not separated by $x_n$, we let $\anr$ and $\bnr$ denote the points on $C_n$ obtained by reflecting $a_n$ and $b_n$, respectively, along the diameter of $C_n$ through $x_n$. When $C_n$ is a straight line, this is just the reflection about the point $x_n$ on the line. Note that $a_n$ and $\anr$ are separated by $x_n$, and so are $b_n$ and $\anr$. Furthermore, we have
\begin{equation}
\frac{\abs{x_n-a_n}}{\abs{x_n-\anr}}=1, \ \frac{\abs{x_n-b_n}}{\abs{x_n-\bnr}}=1
\end{equation}
for all $n$. In order to prove (3.5), we write the three point ratio as
\begin{equation}t_n'=\frac{\abs{x_n'-a_n'}}{\abs{x_n'-b_n'}}=\frac{\abs{x_n'-a_n'}}{\abs{x_n'-\anr'}}\cdot \frac{\abs{x_n'-\anr'}}{\abs{x_n'-b_n'}}. 
\end{equation}
As in the separated configuration above, we also consider three cases here. However, as noted above, Case 1.3 does not depend on the separation configuration. Therefore, we only need to deal with the remaining two cases.

\vspace{0.3cm}
\begin{case}{2.1} $x_n,a_n,b_n$ are collinear. 
In this case, as in Case 1.1, if we apply Lemma 2 to computing the limits of the two ratios on the right hand side of (3.22), we obtain that
\begin{equation}
   \lim_{n\rightarrow\infty}\frac{\abs{x_n'-a_n'}}{\abs{x_n'-\anr'}}
   =1, \ 
\lim_{n\rightarrow\infty}\frac{\abs{x_n'-\anr'}}{\abs{x_n'-b_n'}}=\lim_{n\rightarrow\infty}\frac{\abs{x_n-\anr}}{\abs{x_n-b_n}}=t. 
\end{equation}
Thus (3.5) follows from (3.22) and (3.23) as desired.

\end{case}
\vspace{0.3cm}
\begin{case}{2.2} Suppose the points $a_n,x_n,b_n$ are not collinear and (3.6) holds. By applying the result from Case 1.2 to the separated configurations $\{x_n, a_n, \anr\}$ and $\{x_n, b_n, \anr\}$, respectively, one concludes that (3.23) holds. Hence (3.5) follows from (3.22) and (3.23) in this case as well.
This completes the proof of Theorem 2.
\end{case}

\subsection{Boundary correspondence of $AS$ embeddings}
We complete the proof of Theorem 1 by establishing the following result.
\begin{thm}
    Let $f$ be an AS embedding of the unit disk onto a Jordan domain $G$. Then the boundary extension of $f$ to $\Ss$ is also $AS$. Moreover $f(\D)$ is a symmetric quasidisk.
\end{thm}
\begin{proof}
    First, we note that any $AS$ embedding of the unit disk is conformal. Thus it has a homeomorphic extension, denoted again by $f$, to the boundary $\Ss$. To show that $f$ is AS on $\Ss$, let $\epsilon>0$ and $t>0$ be given. Then choose $\delta>0$ such that the $AS$ condition (1.4) is satisfied for points in $\D$. We will show that, with the same $\delta$, the $AS$ condition (1.4) is also satisfied for points in $\Ss$  \par
    To proceed let $x,a,b\in\Ss$, such that they are all contained in a ball of radius $\delta$ with
    \[\frac{\abs{x-a}}{\abs{x-b}}\leq t.\]
    Let $r_n$ be a sequence of positive numbers such that $r_n<1$ and $r_n\to 1$ as $n\to \infty$. Furthermore let $x_n=r_nx, a_n=r_na$ and $b_n=r_nb$. Then, it is clear that $x_n,a_n,b_n\in\D$ are  contained in a ball of radius $\delta$ with
    \[
    \frac{\abs{x_n-a_n}}{\abs{x_n-b_n}}=\frac{r_n\abs{x-a}}{r_n\abs{x-b}}=\frac{\abs{x-a}}{\abs{x-b}}\leq t.
    \]
    Thus, by the AS condition (1.4) for $f$ in $\D$, we have 
    \[
    \frac{\abs{f(x_n)-f(a_n)}}{\abs{f(x_n)-f(b_n)}}\leq (1+\epsilon)t
    \]
    for all $n$. Taking $n$ to infinity and using the fact that $f$ is a homeomorphism, we reach the desired result that 
    \[
    \frac{\abs{f(x)-f(a)}}{\abs{f(x)-f(b)}}\leq (1+\epsilon)t.
    \]
    Hence $f|_{\Ss}$ is an $AS$ embedding, and by \cite[Theorem 3.2]{brania2004asymptotically}, we conclude that $f(\Ss)$ is a symmetric quasicircle. 
\end{proof}
Observe that using the same idea as above, one can easily show that the extension of $f$ in Theorem 3 is actually $AS$ on the closed disk $\Dc$. We record this result as a corollary.
\begin{cor}
    If $f$ is an $AS$ embedding of the unit disk $\D$ onto a Jordain domain $G$, then its extension to $\Dc$, is an $AS$ embedding of the closed unit disk $\Dc$ onto $\Bar{G}.$
\end{cor}
\section{Final remarks and open problems}
We conclude this paper with some final remarks and open problems. Let $f$ be an embedding of the unit disk into the complex plane. For $x\in \D$ and $r>0$, set 
\[
H_f(x,r)=\frac{\sup\{\abs{f(z)-f(x)}: \abs{z-x}=r, z\in\D\}}{\inf\{\abs{f(y)-f(x)}: \abs{y-x}=r, y\in \D\}}.
\]
Recall from the metric definition, we say that $f$ is $K$-QC, if there exists a $K<\infty$, so that 
\begin{align}\label{eq:K-QC}
\limsup_{r\to 0}H_f(x,r)\leq K
\end{align}
for all $x\in \D$. Moreover, we say that $f$ is $1$-QC if (\ref{eq:K-QC}) holds with $K=1$. It is well known that this is equivalent to the classical definition of conformal maps in the complex plane. We also note that in this case $\limsup$ can be replaced by the limit. Thus an embedding $f:\D\rightarrow \C$ is conformal if and only if
\begin{align}
\lim_{r\to 0}H_f(x,r)=1   
\end{align}
for all $x\in \D$. In fact, (4.2) holds for 1-QC maps in any metric spaces.

Motivated by this limit characterization of conformal maps, one can introduce the concept of {\it uniform conformality} by requiring that the above limit is achieved uniformly. 
\begin{definition}\label{def:uniformly_conformal}
    An embedding $f$ of the unit disk $\D$ into the complex plane, is uniformly conformal if for any $\epsilon>0$ there exists a $\delta>0$ such that for all $0<r<\delta$ and all $x\in\D$,
    $$H_f(x,r)\leq 1+\epsilon.$$
\end{definition}

Combining several results together, we derive the following corollary. 
\begin{cor}\label{cor_2}
    Let $f: \D\rightarrow G$ be a conformal map of the unit disk onto a Jordan domain.
If $J=\partial G$ is a symmetric quasicircle, then $f$ is uniformly conformal on $\D$.
\end{cor}
\begin{proof} By Theorem 1, $f$ is asymptotically symmetric in $\D$. By letting $t=1$ in the {\it AS} condition (1.4), one deduces that $f$ is uniformly conformal in $\D$.
\end{proof}

It remains open whether the symmetric quasidisk property is also necessary for a conformal map $f: \D\rightarrow G$ to be uniformly conformal. Another related open question is whether the {\it AS} property (1.4) with $t=1$ implies the same property for all $t$. As far as we know, this is open even in the unit disk setting. We hope to explore these and other related problems in another project.
\bibliographystyle{alpha}
\bibliography{References}

\end{document}